\documentclass[12pt]{article}

\usepackage{amsfonts}
\usepackage{amsthm}
\usepackage[mathscr]{eucal}
\usepackage{amsmath}
\usepackage{amssymb}
\usepackage{amscd}

\theoremstyle{plain}
\newtheorem{lemma}{Lemma}[section]
\newtheorem{theorem}[lemma]{Theorem}
\newtheorem{proposition}[lemma]{Proposition}
\newtheorem{corollary}[lemma]{Corollary}

\theoremstyle{definition}
\newtheorem{example}[lemma]{Example}
\newtheorem{definition}[lemma]{Definition}
\newtheorem{remark}[lemma]{Remark}

\numberwithin{equation}{section}
\def\begeq{\stepcounter{lemma}\begin{equation}}

\date{}

\begin{document}

\title{Descent Systems for Bruhat Posets}

\author{Lex E. Renner}
\date{March 2008}
\maketitle

\begin{abstract}
Let $(W,S)$ be a finite Weyl group and let $w\in W$. It is widely
appreciated that the descent set
\[
D(w)=\{s\in S\;|\; l(ws)<l(w)\}
\]
determines a very large and important chapter in the study of Coxeter
groups. In this paper we generalize some of those results to the situation of
the Bruhat poset $W^J$ where $J\subseteq S$. Our main results here
include the identification of a certain subset $S^J\subseteq W^J$
that convincingly plays the role of $S\subseteq W$, at least from the point of
view of descent sets and related geometry. The point here is to
use this resulting {\em descent system} $(W^J,S^J)$ to explicitly encode
some of the geometry and combinatorics that is intrinsic to the poset $W^J$.
In particular, we arrive at the notion of an {\em augmented poset},
and we identify the {\em combinatorially smooth} subsets $J\subseteq S$
that have special geometric significance in terms of a certain corresponding
torus embedding $X(J)$. The theory of $\mathscr{J}$-irreducible monoids provides
an essential tool in arriving at our main results.
\end{abstract}

\vspace{10pt}

\section*{Introduction}

If $(W,S)$ is a Weyl group and $w\in W$, $s\in S$, then
either $ws<w$ or else $w<ws$. Hence we define
\[
D(w)=\{s\in S\;|\; l(ws)<l(w)\},
\]
the {\em descent set} of $w\in W$. This innocuous looking situation
is at the heart of many important results in geometry, combinatorics, group
theory and representation theory.

Evidently, the interest in these objects began with Solomon \cite{S}, who defines
a certain subalgebra $B\subseteq\mathbb{Q}[W]$, and uses it to help understand the
representations of $W.$ The algebra $B$ is often called the {\bf descent algebra} since
it can be defined in terms of descent sets. Brown \cite{Brown} looks at this descent algebra
and reconstitutes it as the semigroup algebra of a certain idempotent (``face") semigroup
associated with the reflection arrangement of $W$.

The numbers $|D(w)|$ can be used to calculate the Betti numbers of the associated
torus embedding $X(\emptyset)$ of $W$. These Betti numbers can be obtained directly from the
$h$-vector of the associated rational, convex polytope. In \cite{Stt1}
Stanley proves that the $h$-vector of any simplicial, convex polytope is a symmetric,
unimodal sequence. Stembridge \cite{Stem} proves that the canonical representation of
$W$ on $H^*(X(\emptyset);\mathbb{Q})$ is a permutation representation and, with the help of
Dolgachev-Lunts \cite{DL}, he  computes this representation. In \cite{Br} Brenti studies
these descent polynomials (i.e. the Poincar\'e  polynomials of $X(\emptyset)$) as analogues
of the {\em Eulerian polynomials}. He also looks at the $q$-analogues of these polynomials.

In the theory of group embeddings $|D(w)|$ is an important ingredient in the calculation
of the Betti numbers of the ``wonderful" compactification of a semisimple group
of adjoint type. See \cite{DP,R1}.

In this paper we expand the entire study to include all Bruhat posets $W^J$,
where $J\subseteq S$. In particular, we study the relationship between $W^J$
and a certain torus embedding $X(J)$. This leads us to the notion of an
{\bf augmented poset} $(W^J,\leq, \{\nu_s\}_{s\in S\setminus J})$. This
ordering on $W^J$ is not the usual Bruhat order on $W^J$. It is quantified
in terms of a certain ``ascent/descent" structure on $W^J$.
Further analysis leads us to the notion of a {\bf descent system} $(W^J,S^J)$.
These descent systems are particularly interesting if $X(J)$ is quasi-smooth
in the sense of Danilov \cite{D}. In the remainder of this paper we refer to this
condition as {\bf rationally smooth}, which is currently the accepted terminology.
For the convenience of the reader we give a precise definition.

\begin{definition} \label{ratsm.def}
Let $X$ be a complex algebraic variety of dimension $n$.
Then $X$ is {\em rationally smooth} at $x\in X$ if there is a neighbourhood
$U$ of $x$ in the complex topology such that, for any $y\in U$,
\[
H^m(X,X\setminus\{y\})=(0)
\]
for $m\neq 2n$ and
\[
H^{2n}(X,X\setminus\{y\})=\mathbb{Q}.
\]
Here $H^*(X)$ denotes the cohomology of $X$ with rational coefficients.
\end{definition}

See \cite{Brion} for a modern account of this key notion, along with some
important results about rationally smooth varieties with torus action. 

The main point of this paper is to identify and study the set
\[
\{J\subseteq S\;|\; X(J)\;\text{is rationally smooth}\}.
\]
 See Theorem \ref{almostsmooth.thm}
and Corollary \ref{howmany.thm} below. The descent system  $(W^J,S^J)$ leads to a useful 
combinatorial analogue of the method of Bialynicki-Birula \cite{BB}. This allows us to 
uncover some of the finer geometry of $X(J)$.
The main results of this paper could be stated entirely in the language of
Weyl groups, root systems and polytopes. However, we were led to these results by trying to
calculate the Betti numbers of a certain class of projective varieties that arise
naturally from the theory of reductive monoids. It turns out that ``step one" of this
monoid problem required that we quantify the Betti numbers of $X(J)$. This eventually requires
that we quantify the ``ascent/descent" structure on $W^J$ for certain $J$. Reductive monoids
are an essential tool in showing us how to do this.

\section{$W$-invariant Polytopes}
\label{one.sec}

Let $V$ be a rational vector space and
let $r : W\to GL(V)$ be the usual reflection representation of the
Weyl group $W$. Along with this goes the {\bf Weyl chamber}
$\mathscr{C}\subseteq V$ and the corresponding set of {\bf simple
reflections} $S\subseteq W$. The Weyl group $W$ is generated by $S$, and $\mathscr{C}$ is
a fundamental domain for the action of $W$ on $V$. See Chapter III of \cite{Hum}
for a detailed discussion of Weyl groups.

Let $\lambda\in\mathscr{C}$. In this section we describe the face lattice $\mathscr{F}_\lambda$
of the polytope
\[
\mathscr{P}_\lambda=Conv(W\cdot\lambda),
\]
the convex hull of $W\cdot\lambda$
in $V$. The face lattice $\mathscr{F}_\lambda$ depends only on $W_\lambda=\{w\in W\;|\;
w(\lambda)=\lambda\}=W_{J}=\langle s\;|\;s\in J\rangle$, where $J=\{s\in S\;|\; s(\lambda)=\lambda\}$.
Thus we describe $\mathscr{F}_\lambda=\mathscr{F}_J$ explicitly in terms of $J\subseteq S$.

Closely associated with these polytopes is a certain class of reductive,
algebraic monoids. We use what is known about this class of monoids
to calculate $\mathscr{F}_J$ in terms of the underlying Dynkin diagram
of $(W,S)$.

We now recall some results first recorded in \cite{PR}. Throughout the
paper we use the language and techniques of linear algebraic monoids.
Unfortunately this theory is not widely appreciated, but luckily the
main results and constructions have recently been assembled in \cite{R2}.
See, especially, Chapters 4, 5, 7, and 8 of \cite{R2}. Throughout the paper we
work over the field $\mathbb{C}$ of complex numbers. That is, all algebraic groups
and monoids are assumed to be algebraic varieties defined over $\mathbb{C}$.
Much of what is said in this paper is valid over any algebraically closed field.
Restricting our discussion to varieties over the complex numbers is required only
for the sake of Definition \ref{ratsm.def}.
Let $M$ be an irreducible, normal algebraic monoid with reductive unit
group $G$. We refer to such monoids as {\bf reductive}. The reader
can find any unproved statements about reductive monoids in \cite{P,R2}.
See Solomon's survey \cite{Sol2} for a less technical introduction to the
subject.

If $M$ is a reductive monoid with unit group $G$ we let $B\subseteq G$ be a Borel
subgroup of $G$ and $T\subseteq B$ a maximal torus of $G$. We let $\overline{T}$ denote the
Zariski closure of $T$ in $M$. By part b) of Theorem 4.5 of \cite{R2}, $\overline{T}$ is a
normal, affine torus embedding. The set of {\bf idempotents} $E(\overline{T})$ of
$\overline{T}$ is defined to be
\[
E(\overline{T})=\{e\in\overline{T}\;|\; e^2=e\}.
\]
There is exactly one idempotent in each $T$-orbit on $\overline{T}$.
In the cases of interest in this paper, $E(\overline{T})\setminus\{0\}$ can be
canonically identified (as a poset) with the face lattice $\mathscr{F}_\lambda$
for appropriate $\lambda\in\mathscr{C}$. It turns out that this poset structure on
$E(\overline{T})$ is given by
\[
e \leq f\;\text{if}\; ef=e.
\]
We note that $e\leq f$ if and only if $eT\subseteq\overline{fT}$.
We let $E_1=E_1(\overline{T})=\{e\in E(\overline{T})\;|\; dim(Te)=1\}$. In the
above-mentioned identification, $E_1$ is identified with the vertices of
$\mathscr{F}_\lambda$. We shall see that the combinatorial structure of $E_1$
is much richer because $\overline{T}$ comes from the reductive monoid $M$.

The $G\times G$-orbits of $M$ are particularly important throughout this paper.
Let
\[
\Lambda=\{e\in\overline{T}\;|\; eB=eBe\}
\]
be the {\bf cross section lattice} of $M$ relative to $T$ and $B$. See Chapter 9 of \cite{P}.
It is a basic fact that
\[
M=\bigsqcup_{e\in\Lambda}GeG,
\]
where $GeG\subseteq\overline{GfG}$ if and only if $ef=e$. See Theorem 4.5 of \cite{R2}.

As above we let $S\subseteq W$ be the set of {\bf simple involutions} of $W$ relative to
$T$ and $B$. We can regard $S$ as the set of vertices of a graph with edges $\{(s,t)\;|\; st\neq ts\}$.
Thus we may speak of the connected components of any subset of $S$.

A reductive monoid $M$ with $0\in M$ is called {\bf $\mathscr{J}$-irreducible}
if $M\backslash\{0\}$ has exactly one minimal $G\times G$-orbit. See \cite{PR},
or Section 7.3 of \cite{R2} for a systematic discussion of this important class of
reductive monoids, and for a proof of the following Theorem.

\begin{theorem}   \label{jirred.thm}
Let $M$ be a reductive monoid. The following are equivalent.
\begin{enumerate}
	\item $M$ is $\mathscr{J}$-irreducible.
	\item There is an irreducible rational representation
	      $\rho : M\to End(V)$ which is finite as a morphism
	      of algebraic varieties.
	\item If $\overline{T}\subseteq M$ is the Zariski closure in $M$ of a maximal
	      torus $T\subseteq G$ then the Weyl group $W$ of $T$ acts transitively
	      on the set of minimal nonzero idempotents of $\overline{T}$.
\end{enumerate}
\end{theorem}
Notice in particular that one can construct, up to finite morphism, all
$\mathscr{J}$-irreducible monoids from irreducible representations of a
semisimple group. Indeed, let $G_0$ be semisimple and let $\rho : G_0\to End(V)$
be an irreducible representation. Define $M_1\subseteq End(V)$ to be the Zariski
closure of $\mathbb{C}^*\rho(G_0)$ where $\mathbb{C}^*\subseteq End(V)$ is the set of homotheties.
Finally let $M(\rho)$ be the normalization of $M_1$. Then, according to
Theorem~\ref{jirred.thm}, $M(\rho)$ is $\mathscr{J}$-irreducible.

By the results of Section 4 of \cite{PR}, if $M$ is $\mathscr{J}$-irreducible, there is a unique,
minimal, nonzero idempotent $e_1\in E(\overline{T})$ such that $e_1B=e_1Be_1$, where $B$ is the
given Borel subgroup containing $T$. If $M$ is $\mathscr{J}$-irreducible we say that
$M$ is {\bf $\mathscr{J}$-irreducible of type $J$} if, for this idempotent $e_1$,
\[
J=\{s\in S\;|\; se_1=e_1s\},
\]
where $S$ is the set of simple involutions relative to $T$ and $B$. The
set $J$ can be determined in terms of any irreducible representation
satisfying condition {\em 2} of Theorem~\ref{jirred.thm}. Indeed, let $\lambda\in X(T)_+$
be any highest weight such that $\{s\in S\;|\; s(\lambda)=\lambda\;\}=J$. Then
$M(\rho_\lambda)$ is $\mathscr{J}$-irreducible of {\bf type} $J$ where $\rho_\lambda$
is the irreducible representation of $G_0$ with highest weight
$\lambda$. The representation $\rho_\lambda$ determines a representation of $M(\rho_\lambda)$
on $V$. Furthermore, any two $\mathscr{J}$-irreducible monoids with a
finite, dominant morphism between them are of the same type.
If $e_1$ is the above-mentioned minimal idempotent then $B^-e_1=e_1B^-e_1$,
where $B^-$ is the Borel subgroup opposite $B$. We observe that $e_1Me_1$ is a reductive
monoid with idempotent set $\{0,e_1\}$ and thus $dim(e_1Me_1)=1$. Hence
$e_1B^-e_1$ is also one-dimensional. Thus there exists a character $\chi :
B^-\to \mathbb{C}^*$ such that $be_1=e_1be_1=\chi(b)e_1$ for all $b\in B^-$. It follows
that $B^-$ acts on $e_1(V)$ by the rule
\[
\rho_\lambda(b)(v)=\rho_\lambda(b)(\rho_\lambda(e_1)(v))=\chi(b)\rho_\lambda(e_1)(v)=\chi(b)v.
\]
Therefore $L=e(V)\subseteq V$
is the unique one-dimensional $\rho_\lambda(B^-)$-stable subspace of $V$ with
weight $\lambda$. In particular, $\chi|T=\lambda$ and
$P=\{g\in G_0\;|\; \rho_\lambda(g)(L)=L\;\}$ is a parabolic subgroup of $G_0$
of type $J$.

We now describe the $G\times G$-orbit structure of a $\mathscr{J}$-irreducible monoid
of type $J\subseteq S$. The following result was first recorded in \cite{PR}.

\begin{theorem} \label{jirredorbit.thm}
Let $M$ be a $\mathscr{J}$-irreducible monoid of type $J\subseteq S$.
\begin{enumerate}
	\item There is a canonical one-to-one order-preserving correspondence between
	      the set of $G\times G$-orbits acting on $M$ and the set of $W$-orbits
	      acting on the set of idempotents of $\overline{T}$. This set is
	      canonically identified with $\Lambda=\{e\in E(\overline{T})\;|\; eB=eBe\}$.
	\item $\Lambda\setminus\{0\}\cong \{I\subseteq S\;|\;\text{no connected component
	      of $I$ is contained entirely in $J$}\}$ in such a way that $e$ corresponds
	      to $I\subseteq S$ if $I=\{s\in S\;|\; se=es\neq e\}$. If we let
	      $\Lambda_2=\{e\in\Lambda\;|\; dim(Te)=2\}$ then this bijection identifies
	      $\Lambda_2$ with $S\setminus J$.
	\item If $e\in\Lambda\backslash\{0\}$ corresponds to $I$,  as in 2 above,
	      then $C_W(e)=W_K$ where $K=I\cup\{s\in J\;|\;st=ts\;\text{for all}\;
	      t\in I\}$.
\end{enumerate}
\end{theorem}

\noindent It is worthwhile to pause and notice that $\Lambda$ is completely
determined by $J$. See \cite{R2} for a systematic discussion of $\mathscr{J}$-irreducible
monoids, in particular Lemma 7.8 of \cite{R2}. Notice also that part {\em 1} of Theorem
\ref{jirredorbit.thm} is true for any reductive monoid. See Theorem 4.5 of \cite{R2}
for more of those details.

Let $M$ be a $\mathscr{J}$-irreducible monoid of type $J\subseteq S$ and assume that
$\rho : M\to End(V)$ is an irreducible representation
which is finite as a morphism. Let $G$ be the unit group of $M$ with maximal torus
$T\subset G$. Then let $G_0$ be the semisimple part of $G$
with maximal torus $T_0=G_0\cap T$, and let $\rho_\lambda=\rho|G_0$,
with highest weight $\lambda\in\mathscr{C}$, the rational
Weyl chamber of $G_0$. Then, as above, $J=\{s\in S\;|\; s(\lambda)=\lambda\;\}$.
Recall the polytope $\mathscr{P}_\lambda=Conv(W\cdot\lambda)$, which is the
convex hull of $W\cdot\lambda$ in $X(T_0)\otimes\mathbb{Q}$, where $X(T_0)$ is the
set of characters of $T_0$.
The following corollary of Theorem \ref{jirredorbit.thm} above describes the
face lattice of $\mathscr{P}_\lambda$ in terms of the Weyl group
$(W,S)$.

\begin{corollary}  \label{polyfromjirr.thm}
Let $W$ be a Weyl group and let $r : W\to GL(V)$
be the usual reflection representation of $W$. Let $\mathscr{C}\subseteq V$
be the rational Weyl chamber and let $\lambda\in\mathscr{C}$. Assume that
$J=\{s\in S\;|\; s(\lambda)=\lambda\}$. Then the set of orbits of $W$ acting on
the face lattice $\mathscr{F}_\lambda$ of $\mathscr{P}_\lambda$ is in
one-to-one correspondence with $\{I\subseteq S\;|\;\text{no connected
component of $I$ is contained entirely in $J$}\}$.
\end{corollary}

The subset $I\subseteq S$ corresponds to the unique face $F\in\mathscr{F}_\lambda$
with $I=\{s\in S\;|\; s(F)=F\;\text{and}\; s|F\neq id\}$ whose relative interior
$F^0$ has nonempty intersection with $\mathscr{C}$. See section 7.2 of \cite{R2}
for a detailed discussion of the relationship between $\Lambda$ and the Weyl chamber.

Let $M$ be a $\mathscr{J}$-irreducible monoid of type $J\subseteq S$ and
let $\overline{T}$ be the closure in $M$ of a maximal torus $T$ of $G$.
By part {\em b)} of Theorem 5.4 of \cite{R2}, $\overline{T}$ is a
normal variety. Define
\[
X(J)=[\overline{T}\backslash\{0\}]/\mathbb{C}^*.
\]
The terminology is justified since $X(J)$ depends only on $J$ and not on $M$ or
$\lambda$. The set of distinct, normal $\mathscr{J}$-irreducible monoids associated with $X(J)$
can be identified with the set $\mathscr{C}^J=\{\lambda\in\mathscr{C}\;|\;C_S(\lambda)=J\}$.
In the case $J=\emptyset$, $X(J)$ is the torus embedding studied in \cite{Br,DL,Stem}.

\section{The Augmented Poset}     \label{three.sec}

In this section we define the {\bf augmented poset} $(W^J,\leq,\{\nu_s\})$
associated with the subset $J$ of $S$. Recall that $W^J\subseteq W$ is the set
of minimal length coset representatives of $W_J$ in $W$, and $\leq$ is the usual Bruhat
ordering on $W_J$.

To achieve our objective we use some techniques from
the theory of linear algebraic  monoids. We use this theory to obtain some
important results relating $W^J$ to a certain finite, partially ordered set
$E_1$ of idempotents. That done, we obtain the desired ``ascent/descent"
structure on the poset $W^J$. See Proposition~\ref{dichotomy.prop}.
Our construction has a fundamental relationship with
the extremely important {\bf descent systems} as discussed in Theorem~\ref{alltherest.thm}
and Section \ref{descent.sec}. The reader who does not want to {\em engage the monoids}
might be able to find his own proofs of Proposition~\ref{dichotomy.prop} and
Theorem~\ref{alltherest.thm} using his favorite techniques. See the table in
Remark~\ref{lexicon.rk} for a handy translation between the monoid
jargon and the Bruhat poset jargon. The theory of reductive monoids serves as an
ideal method to help quantify the combinatorics of $W^J$ in geometric terms.

Let $M$ be a reductive, algebraic monoid with unit group $G$. Let $B\subseteq G$
be a Borel subgroup of $G$ and let $T\subseteq B$ be a maximal torus of $B$.
As before $E(\overline{T})=\{e\in\overline{T}\;|\; e^2=e\}$ and
$E_1(\overline{T})=\{e\in\overline{T}\;|\; e^2=e\;\text{and}\; dim(Te)=1\}$.
As usual, $W$ is the Weyl group of $G$ relative to $T$. The next three technical results
will allow us to find our way to the all-important Theorem \ref{order1vsorder2.thm}.

\begin{lemma}  \label{ordere1.lem}
Let $e\in E_1(\overline{T})$. Then
\[
\overline{eB}\setminus\{0\}=\bigcup_{\tau\in X}e\tau B
\]
where $X=\{\tau\in W\;|\;eB\tau^{-1}e\neq 0\}$.
\end{lemma}
\begin{proof}
We first show that $\overline{eB}\backslash\{0\}\subseteq\bigcup_{\tau\in W}e\tau B$. To
this end, first recall $e_1\in E_1(\overline{T})$, the unique rank-one idempotent such
that $e_1B=e_1Be_1$. Then $e_1G=\bigsqcup_{w\in W}e_1BwB=\cup_{w\in W}e_1wB$, since
$e_1B=e_1Be_1=\mathbb{C}^*e_1$. Thus, if $e=\gamma e_1\gamma^{-1}\in E_1$, one sees that
\[
eG=\gamma e_1G=\bigcup_{w\in W}\gamma e_1BwB=\bigcup_{w\in W}\gamma e_1 wB=\bigcup_{\tau\in W}e\tau B.
\]
Hence $\overline{eB}\setminus\{0\}\subseteq eG\subseteq\bigcup_{\tau\in W}e\tau B$.

Thus it suffices to show that $X=\{\tau\in W\;|\; e\tau\in\overline{eB}\}$.
Suppose then, that $e\tau\in\overline{eB}$. Then
$0\neq e\tau\tau^{-1}e\tau\in\overline{eB}\tau^{-1}e\tau$. Thus
$eB\tau^{-1}e\tau\neq 0$.
Conversely, suppose that $eB\tau^{-1}e\tau\neq 0$. Then there exists $b\in B$
such that $0\neq x=eb\tau^{-1}e\tau$. Then
$0\neq x=ex=x\tau^{-1}e\tau\in eB\tau^{-1}e\tau$. Thus
$e\tau\in \mathbb{C}^*e\tau\subseteq eM\tau^{-1}e\tau=eB\tau^{-1}e\tau\subseteq\overline{eB}$
since $B\tau^{-1}e\tau\subseteq\overline{B}$ .
\end{proof}

\begin{corollary} \label{glop.cor}
Let $e\in E_1(\overline{T})$ and let $f\in E(\overline{T})$. Then
\[
\overline{eB}f=\{0\}\cup(\bigcup_{\tau\in X}e\tau Bf).
\]
\end{corollary}

\begin{proposition}\label{whatise.prop}
The following are equivalent.
\begin{enumerate}
\item $ef=e$, and for all $\tau\in X$ with $\tau^{-1}e\tau\neq e$, $e\tau Bf=0$.
\item $eBf=eBe$.
\end{enumerate}
\end{proposition}
\begin{proof}
Assume {\em 1}. Then, by Corollary~\ref{glop.cor},
\[
\overline{eB}f=\{0\}\cup(\bigcup_{\tau\in X}e\tau Bf).
\]
But, by assumption, $e\tau Bf=0$ whenever $\tau^{-1}e\tau\neq e$. Hence
$\overline{eB}f=\{0\}\cup(\bigcup_{\tau\in Z}e\tau Bf)$,
where $Z=\{\tau\in X\;|\; \tau^{-1}e\tau = e\}$. However, if $\tau^{-1}e\tau = e$
then $e\tau = e$. Thus $\overline{eB}f=\{0\}\cup eBf=\{0\}\cup efBf$,
and this a closed subset of $M$.
Using part(ii) of Corollary 7.2 of \cite{P}, we get $efBf=eC_B(f)$. Thus
$\overline{eB}f=eC_B(f)\cup\{0\}$, and
hence $\overline{eB}f$ is the union of two right $C_B(f)$-orbits, $eC_B(f)$ and $\{0\}$.
By part (i) of Theorem 6.16 of \cite{P}, $C_B(e)$ is a connected group. But it
is also a solvable group. Thus, by Theorem 3.1 of \cite{Gr},
$dim(eBf)=1$ since there exists $h\in \mathbb{C}[\overline{eB}f]$ such that $\{0\}=h^{-1}(0)$.
Since $eBe\subseteq eBf$, it follows that $eBe=eBf$.

Conversely, assume {\em 2}. Thus $\overline{eB}f=\overline{eBe}=\{0\}\cup eBe$. But from
Lemma~\ref{ordere1.lem} $\overline{eB}f=\{0\}\cup(\bigcup_{\tau\in X}e\tau Bf)$. Assume that
$e\tau Bf\neq 0$. Then we have
\[
\emptyset\neq e\tau Bf\backslash\{0\}\subseteq\overline{eBe}\backslash\{0\}=eBe=\mathbb{C}^*e.
\]
Thus,
\[
e\in e\tau Bf\subseteq e\tau BfB\subseteq\overline{e\tau B}
\]
since $BfB\subseteq\overline{B}$. But $e\tau B\subseteq\overline{eB}$ and thus
$e\tau B=eB$. Hence $e\tau=e$ and finally $\tau^{-1}e\tau=e$
\end{proof}

\begin{definition}    \label{theorderone1.def}
Let $e,e'\in E_1(\overline{T})$. We say that $e<e'$ if $eBe'\neq 0$ and
$e\neq e'$.
\end{definition}

We shall see in Proposition~\ref{geqvsbeg.prop} that $e<e'$ if and only if
$\overline{BeG}\subsetneq\overline{Be'G}$. Then, in Theorem \ref{order1vsorder2.thm},
we relate this to the Bruhat ordering on $W^J$, where $W_J$ is the centralizer in
$W$ of $e_1$.

\begin{theorem}   \label{orderone1.thm}
Let $e\in E_1$ and let $f\in E$. The following are equivalent.
\begin{enumerate}
\item $eBf=eBe$.
\item
 \begin{enumerate}
  \item $ef=e$.
  \item If $e<e'$ then $e'Bf=0$.
 \end{enumerate}
\item
  \begin{enumerate}
  \item $ef=e$.
  \item If $e<e'$ then $e'f=0$.
 \end{enumerate}
\end{enumerate}
\end{theorem}
\begin{proof}
The equivalence of {\em 1} and {\em 2} is a reformulation of
Proposition~\ref{whatise.prop},
taking into account Definition~\ref{theorderone1.def}. That {\em 2} implies
{\em 3} is obvious. So we assume {\em 3} and then deduce {\em 1}.
By Lemma~\ref{ordere1.lem}
\[
\overline{eB}\backslash\{0\}=\bigcup_{\tau\in X}e\tau B
\]
where $X=\{\tau\in W\;|\;eB\tau^{-1}e\neq 0\}$. Now $ef=e$ so that $eBf=efBf$. Thus
$\overline{eB}f = e\overline{fBf} = e\overline{C}f$, where $C=C_B(f)$.
(again using part(ii) of Corollary 7.2 of \cite{P}) Thus, again by
Proposition~\ref{whatise.prop},
\[
e\overline{C}f\backslash\{0\}=\bigcup_{\gamma\in Y}e\gamma Cf
\]
where $Y=\{\gamma\in W\;|\;eC\gamma^{-1}e\gamma\neq 0\}$. But if $eCe'\neq 0$ then
$eBe'\neq 0$ and then, by assumption, $e'f=0$ as long as $e'\neq e$. Hence $e\gamma f=0$
if $\gamma^{-1}e\gamma\neq e$, and thus $e\gamma Cf=e\gamma fC=0$ for $\gamma\in Y$.
Thus $eBf=\{0\}\cup eCf$, which (as in the proof of Proposition~\ref{whatise.prop})
is one-dimensional. Thus $eBf=eBe$.
\end{proof}

Notice how Theorem \ref{orderone1.thm} allows us to describe the relationship of $B$
and $E$ in terms of the ordering $<$ on $E_1$.

\begin{definition}      \label{bbforeoftbar.def}
Let $e\in E_1$. Define
\[
\mathscr{C}_e=\{f\in E(\overline{T})\;|\; fe=e\;\text{and}\;
fe'=0\;\text{for all}\; e'>e\}.
\]
\end{definition}
\noindent From Theorem~\ref{orderone1.thm}
\[
E(\overline{T})\setminus\{0\} = \bigsqcup_{e\in E_1}\mathscr{C}_e.
\]
The reader is encouraged to think of $\mathscr{C}_e\subseteq E(\overline{T})\setminus\{0\}$
as the combinatorial analogue of a $BB$-cell \cite{BB}.

We recall now the {\bf Gauss-Jordan} elements of $M$. First let
$\mathscr{R}=\{x\in M\;|\; Tx = xT\}/T$. By the results of \cite{R0},
$\mathscr{R}$ is a finite inverse monoid. Furthermore, there is a disjoint union
decomposition
\[
M=\bigsqcup_{r\in\mathscr{R}}BrB.
\]
This {\bf monoid Bruhat decomposition} is discussed in detail in Chapter 8
of \cite{R2}. It results in a perfect analogue, for reductive monoids,
of the much-studied Bruhat decomposition of algebraic groups.

\begin{definition}   \label{gaussjordan.def}
The set of {\em Gauss-Jordan} elements of $\mathscr{R}$ is defined to be
\[
GJ=\{r\in\mathscr{R}\;|\;rB\subseteq Br\}.
\]
\end{definition}

\noindent The following crucial properties of $GJ$ are discussed in \cite{R0}.
\begin{enumerate}
	\item $GJ\cdot W=\mathscr{R}$.
	\item For each $x\in\mathscr{R}$, $GJ\cap xW$ is a singleton.
	\item $GJ$ is a submonoid of $\mathscr{R}$.
	\item $M=\bigsqcup_{r\in GJ}BrG$.
\end{enumerate}
The reader should think of the set of Gauss-Jordan elements of
$\mathscr{R}$ as providing a combinatorial structure to the
(generalized) Gauss-Jordan column-reduction algorithm. If $M$ is the
reductive monoid of $n\times n$ matrices then one can check that,
(relative to $T$ and $B$ the diagonal and upper-triangular matrices, respectively)
$GJ$ can be identified with the set of $0-1$ matrices, in reduced column
echelon form, with at most one non zero entry in each row and column.
See Section 8.3 of \cite{R2} for a detailed discussion of Gauss-Jordan
elements for reductive monoids.

\begin{proposition}  \label{gjorder.prop}
The following are equivalent for $r, s\in GJ$.
\begin{enumerate}
\item $BrG\subseteq\overline{BsG}$.
\item $Br\subseteq\overline{Bs}$.
\end{enumerate}
\end{proposition}
\begin{proof}
The case ``{\em 2} implies {\em 1}" is clear. To prove ``{\em 1} implies {\em 2}"
we shall use the fact that $B\backslash G$ is a complete variety.
Since $s\in GJ$ we have that $BsB=Bs$. Thus $\overline{Bs}B=\overline{Bs}$. But
then, by a result of Steinberg, $\overline{Bs}G=\overline{BsG}$
since $B\backslash G$ is a complete variety. Thus the assumption of {\em 1}
is equivalent to saying that  $BrG\subseteq\overline{Bs}G$.
Thus we can write $r=yg^{-1}$ where $y\in\overline{Bs}$ and $g\in G$.
Hence $rg\in\overline{Bs}$. Thus $BrgB\subseteq\overline{Bs}$. But
$BrgB=BrBgB=BrBwB$ for some $w\in W$. But $1\in\overline{BwB}$, and consequently
$BrB\subseteq\overline{BrBwB}$. We conclude that $BrB\subseteq\overline{Bs}$.
\end{proof}

\noindent Recall that, for $J\subseteq S$,
\[
W^J=\{t\in W\;|\; t\;\text{has minimal length in}\; tW_J\}.
\]
Define also
\[
^JW=\{t\in W\;|\; t\;\text{has minimal length in}\; W_Jt\}.
\]
These will be required in the proof of the following theorem.

\begin{theorem}  \label{gjorder.thm}
Let $r=ve_1,s=we_1$ where $v,w\in W^J$. The following are equivalent.
\begin{enumerate}
 \item $r\leq s$ (i.e. $BrB\subset\overline{BsB}$).
 \item $w\leq v$ (i.e. $BwB\subset\overline{BvB}$).
\end{enumerate}
\end{theorem}
\begin{proof}
We apply Corollary 8.35 of \cite{R2}. But we notice first that, in that setup,
$\Lambda$ is $\{e\in E(\overline{T})\;|\; Be=eBe\}$ while in the present
discussion, $\Lambda$ is $\{e\in E(\overline{T})\;|\; eB=eBe\}$. To eliminate
any potential confusion we shall first {\bf restate} Corollary 8.35 using
$\Lambda = \{e\in E(\overline{T})\;|\; eB=eBe\}$.

If $e,f\in\Lambda$ we write
\begin{enumerate}
\item[] $W_{I_1}=\{w\in W\;|\; we=ew=e\}\;\text{and}\;W_{I_2}=\{w\in W\;|\; we=ew\}$,
\end{enumerate}
and
\begin{enumerate}
\item[] $W_{J_1}=\{w\in W\;|\; wf=fw=f\}\;\text{and}\;W_{J_2}=\{w\in W\;|\; wf=fw\}$.
\end{enumerate}
Let $a, b\in\mathscr{R}$. Then $a=y^{-1}ex$ and $b=t^{-1}fu$ where $x\in^{I_1}W$, $y\in^{I_2}W$,
$u\in^{J_1}W$ and $t\in^{J_1}W$ (here $I_1, I_2, J_1, J_2\subseteq S$).
This is the {\bf normal form} for the elements of $\mathscr{R}$
as in Definition 8.34 of \cite{R2}.
Then (from Corollary 8.35 of \cite{R2}) the following are equivalent.
\begin{enumerate}
  \item[i)] $a\leq b$.
  \item[ii)] $ef=e$, and there exists $w\in W_{I_1}W_{J_2}$ such that
             $x\leq wu$ and $wt\leq y$.
\end{enumerate}
In our situation $W_{I_1}=W_{I_2}=W_{J_1}=W_{J_2}$, $x=u=1$ and $e=f=e_1$.
So condition $ii)$
becomes
\begin{enumerate}
  \item[ii)'] There exists $w\in W_{I_1}$ such that
             $1\leq w$ and $wt\leq y$.
\end{enumerate}
which is equivalent to
\begin{enumerate}
  \item[ii)''] $t\leq y$.
\end{enumerate}
since $t\leq wt$ for all $w\in W_{I_1}$.
Now observe that $t\leq y$ if and only if $t^{-1}\leq y^{-1}$,
while $(^IW)^{-1}=W^I$.
Thus the result follows with $v=y^{-1}$ and $w=t^{-1}$.
\end{proof}

Notice that this might appear counterintuitive. Think of $e_1$ as
``large as possible on the left" and that, multiplication
by some $w$ on the left makes the result smaller ``on the left".
Thus, if $w$ is less than $v$, then $ve_1$ is less than $we_1$.

\begin{proposition}  \label{geqvsbeg.prop}
The following are equivalent for $e,f\in E_1$.
\begin{enumerate}
  \item $e<e'$ (in the ordering of Definition~\ref{theorderone1.def} on $E_1$.).
  \item $BeG\subset\overline{Be'G}$.
\end{enumerate}
\end{proposition}
\begin{proof}
If $BeG\subset\overline{Be'G}$ we first observe that
$e\in e\overline{BeG}\neq 0$. But $e\overline{BeG}\subset\overline{Be'G}$,
and thus $eBe'G\neq 0$. Hence $eBe'\neq 0$.

Conversely, if $eBe'\neq 0$ then $eBe'G\neq 0$, and thus $e\overline{Be'G}\neq 0$.
But $eM=eG\cup\{0\}$ since $e\in E_1$. Thus $e\in e\overline{Be'G}=eM$.
But $e\overline{Be'G}\subset\overline{Be'G}$ since $eB\subset\overline{B}$.
Thus $e\in\overline{Be'G}$ and finally $BeG\subset\overline{Be'G}$.
\end{proof}

\begin{remark}  \label{usinggj.rk}
Notice that $BeG=BrG$ for $r\in We_1\cap eW=\{r\}$ (See Section 8.3 of \cite{R2}).
Similarly for $e'$ and $s\in We_1\cap e'W=\{s\}$. Thus an equivalent statement is
``$BrG\subset\overline{BsG}$" for these $r,s\in GJ$.
\end {remark}

The following theorem is the ``bridge" between the monoid geometry and the
Bruhat combinatorics.

\begin{theorem}    \label{order1vsorder2.thm}
The following are equivalent for $v,w\in W^J$.
\begin{enumerate}
  \item $e=ve_1v^{-1}<e'=we_1w^{-1}$ in $(E_1,<)$.
  \item $w<v$ in $(W^J,<)$, the Bruhat ordering on $W^J$.
\end{enumerate}
\end{theorem}
\begin{proof}
By Proposition~\ref{geqvsbeg.prop}, $e<e'$ if and only if
$BeG\subset\overline{Be'G}$. As in Remark \ref{usinggj.rk}, let
$BeG=BrG$ and $Be'G=BsG$ where $r,s\in GJ$.

By Proposition~\ref{gjorder.prop}, $BrG\subset\overline{BsG}$
if and only if $Br\subseteq\overline{Bs}$. Then by
Theorem~\ref{gjorder.thm}, $Br\subseteq\overline{Bs}$ if and only if
$w<v$, where $r=ev=ve_1$, $s=e'w=we_1$ and $v,w\in W^J$.
\end{proof}

\noindent For $e\in E_1(\overline{T})$ we let
\[
\Gamma(e)=\{g\in E_2(\overline{T})\;|\; ge=e,\;\text{and}\; ge'=0\;\text{for all}\; e'>e\}.
\]

\begin{corollary}      \label{e2isarelation.lem}
Let $g\in E_2(\overline{T})$. Suppose that $e, f\in E_1(\overline{T})$ and that $e\neq f$.
Assume that $ge=e$ and $gf=f$. Then either $e>f$ or else $f>e$. In particular
\[
\Gamma(e)=\{g\in E_2(\overline{T})\;|\; ge=e,\;\text{and}\; ge'=e'\;\text{for some}\; e'<e\}.
\]
\end{corollary}
\begin{proof}
Suppose that $e\not> f$. Recall Definition~\ref{bbforeoftbar.def}. Then $g\in\mathscr{C}_{f}$,
since we have that $ge'=0$ for any $e'>f$. In particular, $g\not\in\mathscr{C}_{e}$.
Thus there exists $e'>e$ such that $ge'=e'$. But then $e'=f$ since $g\in E_2$. Thus $f>e$.
\end{proof}

\begin{remark}  \label{e2isarelation.rk}
If we think of $\leq$ as a relation on $E_1$ then
Corollary~\ref{e2isarelation.lem} says that we can regard
$E_2$ as a subrelation of $\leq$. Notice, in particular, that
\[
E_2=\bigsqcup_{e\in E_1}\Gamma(e).
\]
In general we can identify $E_1$ and $E_2$ with the vertices and
edges, respectively, of a certain polytope. See Remark \ref{lexicon.rk}
for a detailed explanation of how this works in the case of a 
$\mathscr{J}$-irreducible monoid.
\end{remark}

We now return to the case of a $\mathscr{J}$-irreducible monoid. This is the case
that is relevant to the discussion of descent systems.
Recall that, in the general case, $\Lambda_2=\{e\in E_2\;|\; eB=eBe\}$.
But if $M$ is $\mathscr{J}$-irreducible, it follows from part {\em 2} of 
Theorem~\ref{jirredorbit.thm} that, there is a canonical bijection
\[
\Lambda_2\cong S\setminus J.
\]
This bijection is defined by
\[
s\leadsto g_s,
\]
where $g_s\in\Lambda_2$ is the unique idempotent such that
\begin{enumerate}
\item $sg_s=g_ss\neq g_s$.
\item $g_sB\subseteq Bg_s$.
\end{enumerate}
See Theorem 4.13 of \cite{PR} for the detailed proof.

Since each $g\in\Gamma(e)$ is conjugate to one and only one
$g_s\in\Lambda_2$ we can write
\[
\Gamma(e)=\bigsqcup_{s\in S\setminus J}\Gamma_s(e),
\]
where
\[
\Gamma_s(e)=\{g\in\Gamma(e)\;|\;g=vg_sv^{-1}\;\text{for some}\;v\in W\}.
\]
We now translate what we have learned from the monoids into
results about Bruhat posets. Theorem~\ref{order1vsorder2.thm} is the main
result here that makes this possible. The following definition is
the key ingredient that unifies our discussion.

\begin{definition}\label{descentset.def}
Let $(W,S)$ be a Weyl group and let $J\subseteq S$ be a proper subset.
Define
\[
S^J=(W_J(S\setminus J)W_J)\cap W^J.
\]
We refer to $(W^J,S^J)$ as the {\em descent system} associated with
$J\subseteq S$.
\end{definition}

\begin{proposition} \label{csvsdescent.prop}
There is a canonical identification
$S^J\cong\{g\in E_2\;|\;ge_1=e_1\}$.
\end{proposition}
\begin{proof}
We first define
\[
\varphi : W_J(S\setminus J)W_J\to E_1
\]
by $\varphi(w)=we_1w^{-1}$. Then $\varphi(w)=\varphi(v)$ if and only if
$wW_J=vW_J$. Hence $\varphi$ induces an injection $\varphi : S^J\to E_1$.
We now identify the image of $\varphi$. Let
\[
N(e_1)=\{e\in E_1\;|\; ge=e\neq e_1\;\text{and}\;ge_1=e_1\;\text{for some}\;g\in
E_2\}
\]
and let $e\in E_1(e_1)$. Then there exists a {\bf unique} $g\in E_2$ such that
$ge=e$ and $ge_1=e_1$. By Proposition 6.27 of \cite{P} and Theorem 4.13 of \cite{PR}
there exists $u\in W_J$ such that
\[
g=ug_su^{-1}
\]
for some unique $s\in S\setminus J$. But then $use_1su^{-1}=e$, since
$gf=f$ for exactly two rank-one idempotents $f$. It follows that
\[
image(\varphi)=N(e_1).
\]
The sought-after identification, $\theta : S^J\cong\{g\in E_2\;|\;ge_1=e_1\}$, is
defined by
\[
\theta(w)=[e_1,\varphi(w)],
\]
where, by definition, $[e_1,\varphi(w)]$ is the unique rank-two idempotent
$g$ such that $ge_1=e_1$ and $g\varphi(w)=\varphi(w)$.
\end{proof}

\begin{proposition} \label{dichotomy.prop}
Let $u,v\in W^J$ be such that $u^{-1}v\in S^JW_J$. In
particular, $u\neq v$. Then either
$u<v$ or $v<u$ in the Bruhat order $<$ on $W^J$.
\end{proposition}
\begin{proof}
If $u,v\in W^J$ with $v=urc$, $r\in S^J$, $c\in W_J$, consider
as in Proposition~\ref{csvsdescent.prop}, $g_r\in E_2(e_1)$. Then let
$g=ug_ru^{-1}$. Then $g$ is the unique rank-two idempotent such that
$gue_1u^{-1}=ue_1u^{-1}$ and $gve_1v^{-1}=ve_1v^{-1}$.

Recall from Theorem~\ref{order1vsorder2.thm} that, for $u,v\in W^J$
\[
ue_1u^{-1}>ve_1v^{-1}\;\text{if and only if}\; u<v.
\]
But from Corollary~\ref{e2isarelation.lem}, for $g\in E_2$
with $ge_i=e_i$, $i=2,3$, either $e_2>e_3$ or else $e_3>e_2$.
The conclusion follows.
\end{proof}

We let
\[
S^J_s=W_JsW_J\cap W^J.
\]

\begin{remark}  \label{sjs.rk}
Notice that 
\[
S^J=\bigsqcup_{s\in S\setminus J}S^J_s.
\]
Indeed, by the proof of Proposition \ref{csvsdescent.prop}, 
$\theta : S^J\cong\{g\in E_2\;|\;ge_1=e_1\}$. Under this correspondence
$S^J_s$ corresponds to 
$\{g\in E_2\;|\;ge_1=e_1\;\text{and}\;g=wg_sw^{-1}\;\text{for some}\; w\in W_J\}$.
\end{remark}

\begin{definition} \label{descentsets.def}
Let $w\in W^J$. Define
\begin{enumerate}
\item $D^J_s(w)=\{r\in S^J_s\;|\;wrc<w\;\text{for some}\;c\in W_J\}$, and
\item $A^J_s(w)=\{r\in S^J_s\;|\;w<wr\}$.
\end{enumerate}
We refer to $D^J(w)=\sqcup_{s\in S\setminus J}D^J_s(w)$ as the {\em descent set} of $w$
relative to $J$, and $A^J(w)=\sqcup_{s\in S\setminus J}A^J_s(w)$ as the {\em ascent set}
of $w$ relative to $J$.
\end{definition}

By Proposition~\ref{dichotomy.prop}, for any $w\in W^J$, $S^J=D^J(w)\sqcup
A^J(w)$.

\begin{remark}  \label{minlength.rk}
Notice that $wrc<w$ for some $c\in W_J$ if and only if $(wr)_0<w$, where
$(wr)_0\in wrW_J$ is the element of minimal length in $wrW_J$. See
Example~\ref{theeasyeg.ex} for a revealing illustration of the fact that
$S^J=D^J(w)\sqcup A^J(w)$.
\end{remark}

\begin{definition}   \label{augmented.def}
For each $v\in W^J$ and each $s\in S\setminus J$ define
$\nu_s(v)=|A^J_s(v)|$. We refer to $(W^J,\leq,\{\nu_s\})$ as the
{\em augmented poset} of $J$. For convenience we let
\[
\nu(v)=\sum_{s\in S\setminus J}\nu_s(v).
\]
\end{definition}

\begin{example} \label{augposetfors4.ex}
Let $(W,S)$ be the Weyl group of type $A_3$, so that $W = S_4$ and
$S = \{s_1,s_2,s_3\}$. Let $J=\phi$ and write $\nu_i$ for
$\nu_{s_i}$. To keep track of all the numbers $\{\nu_i(w)\;|\; w\in W\}$
define
\[
H(t_1,t_2,t_3)=\sum_{w\in W}t_1^{\nu_1(w)}t_2^{\nu_2(w)}t_3^{\nu_3(w)}.
\]
A straightforward calculation yields
\[
H(t_1,t_2,t_3)=1+(3t_1+5t_2+3t_3)+(3t_2t_3+5t_1t_3+3t_1t_2)+t_1t_2t_3.
\]
\end{example}

See Examples \ref{theeasyeg.ex} and \ref{thategfromwgspbppp.ex} below for a better
illustration of how it works if $J\neq\emptyset$.

\begin{theorem}  \label{alltherest.thm}
Let $J\subset S$ be any proper subset. For $e=ue_1u^{-1}$, $u\in W$, we
write $e=e_u$.
\begin{enumerate}
	\item $E_2\cong\{(u,v)\in W^J\times W^J\;|\;
	u<v\;\text{and}\;u^{-1}v\in S^JW_J\}$.
	\item Let $u\in W^J$ and $e_u=ueu^{-1}\in E_1$. Then\\
        $E_2(e_u)\cong\{v\in W^J\;|\;u^{-1}v\in S^JW_J\}$.
	\item Let $u\in W^J$ and $e_u=ueu^{-1}\in E_1$. Then\\
  $\Gamma(e_u)\cong\{v\in W^J\;|\;u<v\;\text{and}\;u^{-1}v\in S^JW_J\}\cong A^J(u)$.
        \item Let $u\in W^J$ and $e_u=ueu^{-1}\in E_1$. Then\\
  $\Gamma_s(e_u)\cong\{v\in W^J\;|\;u<v\;\text{and}\;u^{-1}v\in S^J_sW_J\}\cong A^J_s(u)$.
        \item If $w\in W^J$ and $s\in S\setminus J$ then $\nu_s(u)=|\Gamma_s(e_u)|$.
\end{enumerate}
\end{theorem}
\begin{proof}
This follows from Proposition~\ref{csvsdescent.prop} and
Proposition~\ref{dichotomy.prop}.
\end{proof}

\begin{remark}   \label{lexicon.rk}
The following table provides the reader with a summary-translation
between the monoid jargon and the Bruhat poset jargon. Let $M$ be a $\mathscr{J}$-irreducible 
monoid of type $J$, and let $\overline{T}$ be the closure in $M$ of a maximal torus.
Let $E=E(\overline{T})$ be
the set of idempotents of $\overline{T}$ and let $E_i=\{f\in E\;|\; dim(fT)=i\}\subset E$.
As above, we let $e_1\in E_1=E_1(\overline{T})$ be the unique element such that $e_1B=e_1Be_1$.
For $e,e'\in E_1$ let $v,w\in W^J$ be the unique elements such that $e=ve_1v^{-1}$ and $e'=we_1w^{-1}$.
We write $e=e_v$ and $e'=e_w$. For $e,f\in E$ we write $e\sim f$ if there exists $w\in W$ such
that $wew^{-1}=f$. If $s\in S\setminus J$ let $g_s\in E_2$ be the unique idempotent such that
$g_ss=sg_s$ and $g_sB=g_sBg_s$. Let $\Lambda^\times=\{I\subset S\;|\; \text{no component of $I$ is
contained in $J$}\}$ and for $I\in \Lambda^\times$ let $I^*=I\cup\{t\in J\;|\; ts=st\;\text{for all}\;
s\in I\;\}$.\\

\begin{tabular}{ | l | l | }
 \hline
 {\bf Reductive Monoid Jargon} & {\bf Bruhat Order Jargon}\\ \hline\hline
 $e_1\in\Lambda_1=\{e_1\}$                  &   $1\in W^J$                           \\ \hline
 $e=e_v\in E_1$                             &   The $v\in W^J$ with $e=ve_1v^{-1}$   \\ \hline
 $e_v\leq e_w$ in $E_1$, i.e. $e_vBe_w\neq 0$   & $w\leq v$ in $W^J$                 \\ \hline
 \;                                         &   $(u,v)\in W^J\times W^J$ such that\\
 $E_2=\{g\in E\;|\; dim(gT)=2\}$            &   $u < v$ and $u^{-1}v\in S^JW_J$      \\ \hline
 $\{g\in E_2\;|\; gB=gBg\}$                 &   $S\setminus J$                       \\ \hline
 $\{g\in E_2\:|\; ge_1=e_1\;\}$             &   $S^J=(W_J(S\setminus J)W_J)\cap W^J$ \\ \hline
 $\{g\in E_2\:|\; ge_1=e_1, g\sim g_s\}$    &   $S^J_s=(W_JsW_J)\cap W^J$            \\ \hline
 $E_2(e_w)=\{g\in E_2\;|\; ge_w=e_w\}$      &   $\{v\in W^J\;|\;w^{-1}v\in S^JW_J\}$ \\ \hline
 $\Gamma(e_w)=\{g\in E_2(e_w)\;|\; ge'=e'\;\text{for some}\; e'<e_w\}$ & $A^J(w)=\{r\in S^J\;|\;w<wr\}$   \\ \hline
 $\Gamma_s(e_w)=\Gamma(e_w)\cap\{g\in E_2\;|\; g\sim g_s\;\}$  &  $A^J_s(w)=\{r\in S^J_s\;|\;w<wr\}$   \\ \hline
 $E(\overline{T})\setminus\{0\}$            &   $\{(w,I)\;|\; I\in\Lambda^\times,\; w<ws\;\text{if}\; s\in I^*\}$ \\ \hline
\end{tabular}
\end{remark}

\noindent The ``picture" here is this. The subset $W^J\subseteq W$ is canonically
identified with the set of vertices of the rational polytope $\mathscr{P}_\lambda$. On the
other hand there is a canonical ordering on $E_1=E_1(\overline{T})$ coming from the
associated reductive monoid. Evidently $(E_1,\leq)$ and $(W^J,\leq)$ are anti-isomorphic as
posets. Furthermore the set of edges $Edg(\mathscr{P}_\lambda)$ of $\mathscr{P}_\lambda$
is canonically identified with $E_2=E_2(\overline{T})$. If $g(v,w)=g(w,v)\in
Edg(\mathscr{P}_\lambda)$ is the edge of $\mathscr{P}_\lambda$ joining the distinct
vertices $v,w\in W^J$ then either $v<w$ or else $w<v$. Given $v\in W^J$, with edges
$Edg(v)=\{g\in E_2\;|\; g=g(v,w)\;\text{for some}\; w\in W^J\}$, the question of
whether $v<w$ or $w<v$ is coded in the ``descent system" $(W^J, S^J)$.

\section{Bruhat Posets and Simple Polytopes}       \label{two.sec}

Recall that if $\lambda\in\mathscr{C}$, then the rational polytope $\mathscr{P}_\lambda$
records the combinatorial properties of the orbit structure of $T$ on $\overline{T}$.
In this section we characterize, in terms of $J\subseteq S$, the conditions under
which $\mathscr{P}_\lambda$ is a {\bf simple polytope}. A polytope $\mathscr{P}$ is called
{\em simple} if each vertex figure of $\mathscr{P}$ is a simplex, or equivalently,
each vertex is the endpoint of exactly $m$ edges $\mathscr{P}$, where $m$ is the
dimension of $\mathscr{P}$. An equivalent formulation is the following. Recall that
$X(J)=[\overline{T}\setminus\{0\}]/\mathbb{C}^*$, where $\overline{T}$ is as above, with
$J=\{s\in S\;|\; s(\lambda)=\lambda\}$. Then $\mathscr{P}_\lambda$ is a simple polytope
if and only if $X(J)$ is rationally smooth.

\begin{definition}   \label{combsmooth.def}
We refer to $J$ as {\em combinatorially smooth} if
$\mathscr{P}_\lambda$ is a simple polytope.
\end{definition}

As usual we let $e_1\in E_1$ be the unique rank-one idempotent such that $e_1B=e_1Be_1$.
If $J\subseteq S$ we let $\pi_0(J)$ denote the set of connected components of $J$.
To be more precise, let $s,t\in J$. Then $s$ and $t$ are in the same connected component
of $J$ if there exist $s_1,...,s_k\in J$ such that $ss_1\neq s_1s$, $s_1s_2\neq s_2s_1$,....,
$s_{k-1}s_k\neq s_ks_{k-1}$, and $s_kt\neq ts_k$.

The following theorem indicates exactly how to detect the very
interesting condition of Definiton~\ref{combsmooth.def}. We use, without mention,
the natural correspondence between the face lattice of $\mathscr{P}_\lambda$
and the set of idempotents of $\overline{T}$. See Remark \ref{lexicon.rk}.

\begin{theorem}     \label{almostsmooth.thm}
Let $\lambda\in\mathscr{C}$. The following are equivalent.
\begin{enumerate}
\item $\mathscr{P}_\lambda$ is a simple polytope.
\item There are exactly $|S|$ edges of $\mathscr{P}_\lambda$ meeting at $\lambda$.
\item $J=\{s\in S\;|\; s(\lambda)=\lambda\}$ has the properties
  \begin{enumerate}
    \item If $s\in S\backslash J$, and $J\not\subseteq C_W(s)$, then there is a unique
    $t\in J$ such that $st\neq ts$. If $C\in\pi_0(J)$ is the unique
    connected component of $J$ with $t\in C$ then $C\backslash\{t\}\subseteq C$
    is a setup of type $A_{l-1}\subseteq A_l$.
    \item For each $C\in\pi_0(J)$ there is a unique $s\in S\backslash J$
    such that $st\neq ts$ for some $t\in C$.
  \end{enumerate}
\end{enumerate}
\end{theorem}
\begin{proof}
{\em 1} and {\em 2} are equivalent by standard results about polytopes.

Assume that {\em 3} holds. We now deduce from this that {\em 2} holds. This is
equivalent to the statement $|\{f\in E_2(\overline{T})\;|\;fe_1=e_1\}|=|S|$.
Let $\Lambda_2=\{f\in\Lambda\;|\; dim(fT)=2\}$, and recall that
\[
\Lambda_2\cong S\backslash J
\]
via the correspondence $f=f_s$ if $sf=fs\neq f$. See Theorem 4.16 of \cite{PR}.
So we write
\[
\Lambda_2=\{f_s\;|\; s\in S\backslash J\}.
\]
Then from part (iii) of Proposition 6.27 of \cite{P}
\[
\{f\in E_2(\overline{T})\;|\;fe_1=e_1\}=\bigcup_{w\in W_J}w\Lambda_2w^{-1}
=\bigsqcup_{s\in S\backslash J}Cl_{W_J}(f_s)
\]
where $Cl_{W_J}(f_s)$ is the $W_J$-conjugacy class of $f_s$.
Let $s\in S\backslash J$.\\
{\em Case 1:}  $st=ts$ for all $t\in J$.\\
Then $f_sw=wf_s$ for all $w\in W_J$. In this case $Cl_{W_J}(f_s)=\{f_s\}$.\\
{\em Case 2:} $ts\neq st$ for some unique $t\in J$. Let $C$ be that unique
connected component of $J$ with $t\in C$. Thus $C_{W_J}(f_s)=W_{J\backslash\{t\}}$ and,
consequently, $Cl_{W_J}(f_s)\cong W_J/W_{J\backslash\{t\}}$. But,
by part {\em (a)} of the assumption,
\[
W_J/W_{J\backslash\{t\}}\cong W_C/W_{C\backslash\{t\}}\cong S_{m+1}/S_m,
\]
where $|C|=m$ and $S_m$ is the symmetric group on $m$ letters. Thus
\[
|Cl_{W_J}(f_s)|=|S_{m+1}/S_m|=\frac{(m+1)!}{m!}=m+1.
\]
Since, by assumption, each $C$ occurs for exactly one $t\in S\backslash J$,
we conclude that
\[
|\{f\in E_2(\overline{T})\;|\;fe_1=e_1\}|=\left[\sum_{C\in\pi_0(J)}(|C|+1)\right]+|\{s\in
S\backslash J\;|\; st=ts\;\text{for all}\; t\in J\}|.
\]
But $\sum_{C\in\pi_0(J)}(|C|+1)=|J|+|\pi_0(J)|$ while $|\{s\in
S\backslash J\;|\; st=ts\;\text{for all}\; t\in J\}|=|S\backslash J|-
|\pi_0(J)|$. Thus, $|\{f\in E_2(\overline{T})\;|\;fe_1=e_1\}|=|S|$.

Assume {\em 2}, and let $s\in S\backslash J$. As above,
\[
\{f\in E_2(\overline{T})\;|\;fe_1=e_1\}
=\bigsqcup_{s\in S\backslash J}Cl_{W_J}(f_s).
\]
If $s\in S\setminus J$ there are two cases.\\
{\bf Case 1:} $st=ts$ for all $t\in J$.\\
In this case $Cl_{W_J}(f_s)=\{f_s\}$.\\
{\bf Case 2:} $st\neq ts$ for some $t\in J$.\\
For each such $t$ there is a unique $C\in \pi_0(J)$ such that $t\in C$.
This is because the connected components of $J$ are disjoint.

One then checks that,
\[
C_{W_C}(f_s) = W_{C\setminus\{t\}},
\]
where $t\in C$ is the unique element such that $st\neq ts$.
($t$ is unique since $S$ is a tree)

But $W_J=\Pi_{C\in\pi_0(J)}W_C$ and $C_{W_J}(f_s)=\Pi_{C\in\pi_0(J)}C_{W_C}(f_s)$.
Hence
\[
Cl_{W_J}(f_s)=\Pi_{C\in V_s}Cl_{W_C}(f_s)
\]
where $V_s=\{C\in\pi_0(J)\;|\; st\neq ts \;\text{for some}\; t\in C \}$.

Hence, for this $s\in S\setminus J$,
\[
|Cl_{W_J}(f_s)|=\Pi_{C\in V_s}|Cl_{W_C}(f_s)|.
\]
Thus
\[
|Cl_{W_J}(f_s)|=\Pi_{C\in V_s}|W_{C}/W_{C\backslash\{t(s,C)\}}|.
\]
where $t(s,C)$ is the unique element of $C$ that fails to commute with $s\in S\setminus J$.

Combining Case 1 and Case 2, we obtain that
\begin{equation} \label{the.equ}
|\{f\in E_2(\overline{T})\;|\;fe_1=e_1\}|=|Int(S\setminus J)| +
\sum_{s\in Bd(S\setminus J)}\Pi_{C\in V_s}|W_{C}/W_{C\backslash\{t(s,C)\}}|
\end{equation}
where $Int(S\setminus J)=\{v\in S\setminus J\;|\; vt=tv\;\text{for all}\; t\in J\}$ and
$Bd(S\setminus J)=\{v\in S\setminus J\;|\; vt\neq tv\;\text{for some}\; t\in J\}$. Notice that
\[
\pi_0(J)=\bigcup_{s\in Bd(S\setminus J)}V_s
\]
since any connected component $C$ of $J$ contains at least one element that fails to commute
with something in $S\setminus J$. Also it is a basic fact about Weyl groups that, if
$C\subseteq S$ is connected and $t\in C$
then $|W_{C}/W_{C\backslash\{t\}}|\geq |C|+1$, with equality
if and only if $C\backslash\{t\}\subseteq C$ is a setup of type
$A_{l-1}\subseteq A_l$.

One checks that if the right-hand-side of this equation is equal to $|S|$ then
all of the following must hold (since failure of any of them would make the RHS
of (\ref{the.equ}) larger than $|S|$).

\begin{enumerate}
	\item For each $s\in Bd(S\setminus J)$, and for any $C\in V_s$, $C\backslash\{t(s,C)\}\subseteq C$
	is a setup of type $A_{l-1}\subseteq A_l$.
	\item For each $s\in Bd(S\setminus J)$, $V_s$ contains exactly one element.
	\item For distinct elements $r,s\in S\setminus J$,
	$V_s\cap V_r = \emptyset$.
\end{enumerate}
It then follows easily from this, that {\em 3} holds.
\end{proof}

In the next two examples one can use Equation \ref{the.equ} to calculate
$|\{f\in E_2(\overline{T})\;|\;fe_1=e_1\}|$.

\begin{example} \label{comsm1.ex}
Let $(W,S)$ be a Weyl group of type $A_3$, so that
$S=\{r,s,t\}$ with $rs\neq sr$ and $st\neq ts$.
\begin{enumerate}
	\item [(a)]If $J=\{r,t\}$ then $|\{f\in E_2(\overline{T})\;|\;fe_1=e_1\}|=4$. In this example
	$V_s=\{\{r\},\{t\}\}$, which violates condition 2. at the end of the proof of
	Theorem \ref{almostsmooth.thm}.
	\item [(b)]If $J=\{s\}$ then $|\{f\in E_2(\overline{T})\;|\;fe_1=e_1\}|=4$. In this example
	$V_r=V_t=\{\{s\}\}$ which violates condition 3. at the end of the proof of
	Theorem \ref{almostsmooth.thm}.
	\item [(c)]If $J=\emptyset, \{r\}$, or $\{r,s\}$ then
	$|\{f\in E_2(\overline{T})\;|\;fe_1=e_1\}|=3$.
	So these ones are combinatorially smooth.
\end{enumerate}
\end{example}

\begin{example} \label{comsm2.ex}
Let $(W,S)$ be a Weyl group of type $C_3$, so that
$S=\{r,s,t\}$ with $rs\neq sr$ and $st\neq ts$, and
$t$ corresponds to a short root. If $J=\{s,t\}$
then $|\{f\in E_2(\overline{T})\;|\;fe_1=e_1\}|=4$. In this example
$\{t\}\subseteq \{s,t\}$ is a setup of type $A_1\subseteq C_2$ which
violates condition 1. at the end of the proof of Theorem \ref{almostsmooth.thm}.
\end{example}

Notice in particular, if $(W,S)$ is an irreducible Weyl group and
$J\subseteq S$ is a combinatorially smooth subset, then each connected component
of $J$ contains exactly one end-node of $S$.

\begin{corollary}    \label{howmany.thm}
For each irreducible Dynkin diagram we obtain the following calculation for
$\{J\subseteq S\;|\; J\;\text{is combinatorially smooth}\}$.
For each type the list is grouped into the different cases
depending on which of the end-nodes are elements of $J$.
\begin{enumerate}
  \item $A_1$.
  \begin{enumerate}
	\item $J=\phi$.
  \end{enumerate}
        $A_n$, $n\geq 2$. Let $S=\{s_1,...,s_n\}$.
\begin{enumerate}
  \item $J=\phi$.
	\item $J=\{s_1,...,s_i\}$, $1\leq i<n$.
	\item $J=\{s_j,...,s_n\}$, $1<j\leq n$.
	\item $J=\{s_1,...,s_i,s_j,...s_n\}$, $1\leq i$, $i\leq j-3$ and $j\leq n$.
\end{enumerate}
	\item $B_2$.
	\begin{enumerate}
	\item $J=\phi$.
	\item $J=\{s_1\}$.
	\item $J=\{s_2\}$.
\end{enumerate}
	      $B_n$, $n\geq 3$. Let $S=\{s_1,...,s_n\}$, $\alpha_n$ short.
\begin{enumerate}
  \item $J=\phi$.
	\item $J=\{s_1,...,s_i\}$, $1\leq i<n$.
	\item $J=\{s_n\}$.
	\item $J=\{s_1,...,s_i,s_n\}$, $1\leq i$ and $i\leq n-3$.
\end{enumerate}
   \item $C_n$, $n\geq 3$. Let $S=\{s_1,...,s_n\}$, $\alpha_n$ long.
\begin{enumerate}
  \item $J=\phi$.
	\item $J=\{s_1,...,s_i\}$, $1\leq i<n$.
	\item $J=\{s_n\}$.
	\item $J=\{s_1,...,s_i,s_n\}$, $1\leq i$ and $i\leq n-3$.
\end{enumerate}
	\item $D_n$, $n\geq 4$. Let $S=\{s_1,...s_{n-2},s_{n-1},s_n\}$.
\begin{enumerate}
	\item $J=\phi$.
	\item $J=\{s_1,...,s_i\}$, $i\leq n-3$.
	\item $J=\{s_{n-1}\}$.
	\item $J=\{s_n\}$.
	\item $J=\{s_1,...,s_i,s_{n-1}\}$, $i\leq n-4$.
	\item $J=\{s_1,...,s_i,s_n\}$, $i\leq n-4$.
\end{enumerate}
	\item $E_6$. Let $S=\{s_1,s_2,s_3,s_4,s_5,s_6\}$.
\begin{enumerate}
	\item $J=\phi$.
	\item $J=\{s_1\}$ or $\{s_1,s_2\}$.
	\item $J=\{s_5\}$ or $\{s_4,s_5\}$.
	\item $J=\{s_6\}$.
  \item $J=\{s_1,s_5\},\{s_1,s_2,s_5\}$ or $\{s_1,s_4,s_5\}$.
  \item $J=\{s_1,s_6\}$.
  \item $J=\{s_5,s_6\}$
  \item $J=\{s_1,s_5,s_6\}$.
\end{enumerate}
	\item $E_7$. Let $S=\{s_1,s_2,s_3,s_4,s_5,s_6,s_7\}$.
\begin{enumerate}
  \item $J=\phi$.
	\item $J=\{s_1\}, \{s_1,s_2\}$ or $\{s_1,s_2,s_3\}$.
	\item $J=\{s_6\}$ or $\{s_5,s_6\}$.
	\item $J=\{s_7\}$.
  \item $J=\{s_1,s_6\},\{s_1,s_2,s_6\},\{s_1,s_2,s_3,s_6\},\{s_1,s_5,s_6\},$
  or $\{s_1,s_2,s_5,s_6\}$.
  \item $J=\{s_6,s_7\}$.
  \item $J=\{s_1,s_7\}$ or $\{s_1,s_2,s_7\}$.
  \item $J=\{s_1,s_6,s_7\},\{s_1,s_2,s_6,s_7\}$.
\end{enumerate}
	\item $E_8$. Let $S=\{s_1,s_2,s_3,s_4,s_5,s_6,s_7,s_8\}$.
\begin{enumerate}
  \item $J=\phi$.
	\item $J=\{s_1\}, \{s_1,s_2\}, \{s_1,s_2,s_3\}$ or $\{s_1,s_2,s_3,s_4\}$.
	\item $J=\{s_7\}$ or $\{s_6,s_7\}$.
	\item $J=\{s_8\}$.
  \item $J=\{s_1,s_7\},\{s_1,s_2,s_7\},\{s_1,s_2,s_3,s_7\}, \{s_1,s_2,s_3,s_4,s_7\}$,\\
  $\{s_1,s_6,s_7\},\{s_1,s_2,s_6,s_7\}, \{s_1,s_2,s_3,s_6,s_7\}$
  or $\{s_1,s_2,s_5,s_6\}$.
  \item $J=\{s_7,s_8\}$.
  \item $J=\{s_1,s_8\}, \{s_1,s_2,s_8\}$ or $\{s_1,s_2,s_3,s_8\}$.
  \item $J=\{s_1,s_7,s_8\},\{s_1,s_2,s_7,s_8\}$.
\end{enumerate}
	\item $F_4$. Let $S=\{s_1,s_2,s_3,s_4\}$.
\begin{enumerate}
	\item $J=\phi$.
	\item $J=\{s_1\}$ or $\{s_1,s_2\}$.
	\item $J=\{s_4\}$ or $\{s_3,s_4\}$.
	\item $J=\{s_1,s_4\}$.
\end{enumerate}
	\item $G_2$. Let $S=\{s_1,s_2\}$.
\begin{enumerate}
	\item $J=\phi$.
	\item $J=\{s_1\}$.
	\item $J=\{s_2\}$.
\end{enumerate}
\end{enumerate}
\end{corollary}
\begin{proof}
This is an elementary calculation with Dynkin diagrams using
Theorem~\ref{almostsmooth.thm}. The numbering of the elements of $S$
is as follows. For types $A_n, B_n, C_n, F_4,$ and $G_2$ it is the usual
numbering. In these cases the end nodes are $s_1$ and $s_n$. For type
$E_6$ the end nodes are $s_1,s_5$ and $s_6$ with $s_3s_6\neq s_6s_3$. For type
$E_7$ the end nodes are $s_1,s_6$ and $s_7$ with $s_4s_7\neq s_7s_4$. For type
$E_8$ the end nodes are $s_1,s_7$ and $s_8$ with $s_5s_8\neq s_8s_5$.
In each case  of type $E_n$, the nodes corresponding to $s_1, s_2,...,s_{n-1}$
determine the unique subdiagram of type $A_{n-1}$. For type $D_n$ the end
nodes are $s_1,s_{n-1}$ and $s_n$. The two subdiagrams of $D_n$, of type $A_{n-1}$,
correspond to the subsets $\{s_1, s_2,...,s_{n-2},s_{n-1}\}$ and
$\{s_1, s_2,...,s_{n-2},s_n\}$ of $S$.
\end{proof}

\begin{remark}   \label{ratsmdan.rk}
It is easy to check that $J\subseteq S$ is combinatorially smooth if and
only if $X(J)$ is rationally smooth. Indeed, this follows directly from
Corollary 2 on page 136 of \cite{Brion}.
\end{remark}

\section{The Descent System $(W^J,S^J)$}\label{descent.sec}

Let $(W,S)$ be a finite Weyl group and let $w\in W$. It is widely
appreciated \cite{BjBr,Br,S} that the {\bf descent set}
\[
D(w)=\{s\in S\;|\; l(ws)<l(w)\}
\]
determines a very large and important chapter in the study of Coxeter
groups. In this section we interpret the results of Sections~\ref{three.sec}
and ~\ref{two.sec} solely in the language of Coxeter groups applied
to $W$, $W^J$, $J\subseteq S$ and the Bruhat ordering on $W^J$. Our main
result here is the explicit identification of the subset $S^J\subseteq W^J$.

Recall, from Definition \ref{descentset.def}, that
\[
S^J=(W_J(S\setminus J)W_J)\cap W^J.
\]
We refer to $(W^J,S^J)$ as the {\em descent system} associated with
$J\subseteq S$.

\begin{proposition} \label{csvsdescent2.prop}
Let $(W^J,S^J)$ be the descent system associated with $J\subseteq S$.
The following are equivalent.
\begin{enumerate}
 \item $J$ is combinatorially smooth.
 \item $|S^J|=|S|$.
 \item $X(J)$ is rationally smooth.
\end{enumerate}
\end{proposition}
\begin{proof}
The equivalence of {\em 1} and {\em 2} follows from Proposition~\ref{csvsdescent.prop}
using part {\em 2} of Theorem~\ref{almostsmooth.thm}. The equivalence of {\em 1} and {\em 3}
follows from Remark \ref{ratsmdan.rk}.
\end{proof}

Assume that $J\subseteq S$ is combinatorially smooth. Recall that, for $s\in
S\setminus J$,
\[
S^J_s=(W_JsW_J)\cap W^J.
\]
Recall now, that for $s\in S\setminus J$, there is a unique $g_s\in\Lambda_2$
such that $\{s\}=\{t\in S\;|\;tg_s=g_st\neq g_s\}$. Furthermore,
$s\leadsto g_s$ determines a bijection between $S\setminus J$
and $\Lambda_2$. Each $g\in E_2(\overline{T})$ is conjugate to a unique
$g_s$, $s\in S\setminus J$. See part {\em 2} of Theorem~\ref{jirredorbit.thm}.

\begin{theorem} \label{whatissj.thm}
Assume that $J\subseteq S$ is combinatorially smooth. Then
\begin{enumerate}
	\item $S^J=\bigsqcup_{s\in S\setminus J}S^J_s$.
	\item Let $s\in S\setminus J$. In case $st=ts$ for all $t\in J$, $S^J_s=\{s\}$.
	Otherwise, $S^J_s=\{s,t_1s,t_2t_1s,...,t_m\cdots t_2t_1s\}$
	where $C=C_s=\{t_1,t_2,...,t_m\}$, $st_1\neq t_1s$ and $t_it_{i+1}\neq t_{i+1}t_i$
	for for $i=1,...,m-1$.
	\item $S^J_s\cong\{g\in E_2\;|\;ge_1=e_1\;\text{and}\;cgc^{-1}=g_s\;
	\text{for some}\; c\in W_J\}$.
\end{enumerate}
\end{theorem}
\begin{proof}
Part {\em 1} follows from Remark \ref{sjs.rk}.
Part {\em 2} follows from well-known information about the standard inclusion of 
symmetric groups $S_{n}\subseteq S_{n+1}$. See Theorem~\ref{almostsmooth.thm} above.
Part {\em 3} follows from Remark \ref{sjs.rk}. See also the proof of Proposition
\ref{csvsdescent.prop}.
\end{proof}

\begin{example}  \label{theeasyeg.ex}
Let
\[
W=<s_1,...s_n>
\]
be the Weyl group of type $A_n$  (so that $W\cong S_{n+1}$), and let
\[
J=\{s_2,...,s_n\}\subseteq S=\{s_1,...,s_n\}.
\]
Then $J\subseteq S$ is combinatorially smooth. One checks, using
Theorem~\ref{whatissj.thm}, that
\[
W^J=\{1,s_1,s_2s_1,s_3s_2s_1,..., s_ns_{n-1}\cdots s_2s_1\}.
\]
Notice that
\[
1<s_1<s_2s_1<...<s_ns_{n-1}\cdots s_1.
\]
In this very special example we obtain that $S^J=W^J\setminus\{1\}$.
Furthermore,
\[
A^J(w)=A^J_{s_1}(w)
\]
for each $w\in W^J$, since $S\setminus J = \{s_1\}$. Finally we obtain, by calculation, that
\begin{enumerate}
\item[] $(s_j\cdots s_1)(s_1) = [s_j\cdots s_2]$,
\item[] $(s_j\cdots s_1)(s_i\cdots s_1) = (s_{i-1}\cdots s_1)[s_j\cdots s_2]$ if $1<i\leq j$, and
\item[] $(s_j\cdots s_1)(s_i\cdots s_1) = (s_i\cdots s_1)[s_{j+1}\cdots s_2]$ if $i>j\geq 1$.
\end{enumerate}
We conclude from this that
\[
A^J(s_j\cdots s_1)=\{s_m\cdots s_1\;|\; m>j\}.
\]
Let us write $a_j=s_j\cdots s_1$ if we think of $s_j\cdots s_1\in W^J$, and
$r_j=s_j\cdots s_1$ if we think of $s_j\cdots s_1\in S^J$. Also,
if $w\in W$, we write $w_0$ for the element of minimal length in $wW_J$.
By the calculation above we obtain that
\begin{enumerate}
	\item[] $(a_jr_i)_0=1<a_j$ if $1=i\leq j$,
	\item[] $(a_jr_i)_0=a_{i-1}<a_j$ if $1<i\leq j$, and
	\item[] $(a_jr_i)_0=a_i>a_j$ if $i>j$.
\end{enumerate}
\end{example}

\begin{example} \label{thategfromwgspbppp.ex}
Let
\[
W=<s_1,...s_n>
\]
be the Weyl group of type $A_n$ (so that $W\cong S_{n+1}$),
and let
\[
J=\{s_3,...s_n\}\subseteq S.
\]
Notice that $J\subseteq S$ is combinatorially smooth.

If $w\in W^J$ then $w=a_p$, $w=b_q$, or else  $w=a_pb_q$.
Here $a_p=s_p\cdots s_1$ ($1\leq p\leq n$) and $b_q=s_q\cdots s_2$ ($2\leq q\leq n$).
If we adopt the useful convention $a_0=1$ and $b_1=1$, then we can write
\[
W^J=\{a_pb_q\;|\;0\leq p\leq n\;\text{and}\;1\leq q\leq n\}
\]
with uniqueness of decomposition. Let $w=a_pb_q\in W^J$. After some tedious
calculation with braid relations and reflections, we obtain that
\begin{enumerate}
	\item[a)] $A^J_{s_1}(a_pb_q)=\{s_1\}$ if $p < q$.\\
	      $A^J_{s_1}(a_pb_q)=\emptyset$ if $q \leq p$.\\
	      Thus $\nu_{s_1}(a_pb_q)=1$ if $p < q$ and $\nu_{s_1}(a_pb_q)=0$ if $q \leq p$.
	\item[b)] $A^J_{s_2}(a_pb_q)=\{s_m\cdots s_n\;|\; m>q\}$ if $q<n$.\\
	      $A^J_{s_2}(a_pb_q)=\emptyset$ if $q=n$.\\
	      Thus $\nu_{s_2}(a_pb_q)=n-q$.
\end{enumerate}
It is interesting to compute the two-parameter ``Euler polynomial"
\[
H(t_1,t_2)=\sum_{w\in W^J}t_1^{\nu_1(w)}t_2^{\nu_2(w)}
\]
of the augmented
poset $(W^J,\leq,\{\nu_1,\nu_2\})$ (where we write $\nu_i$ for $\nu_{s_i}$). A
simple calculation yields
\[
H(t_1,t_2)=\sum_{k=1}^n[kt_1+(n+1-k)]t_2^{n-k}.
\]
\end{example}

\vspace{20pt}

\noindent Lex E. Renner \\
Department of Mathematics \\
University of Western Ontario \\
London, N6A 5B7, Canada \\

\enddocument